\documentclass[11pt]{article}
\usepackage{latexsym}
\usepackage{amssymb}
\usepackage{amsmath}
\usepackage[all]{xy}

\newtheorem{teo}{Theorem}[section]
\newtheorem{prop}[teo]{Proposition}
\newtheorem{lemma}[teo] {Lemma}
\newtheorem{coro}[teo]{Corollary}

\newtheorem{example}{Example}
\newtheorem{remark}{Remark}

\def\qed{\hfill \mbox{$\square$}}

\def\Hom{\mathop{\rm Hom}\nolimits}

\def\Ext{\mathop{\rm Ext}\nolimits}

\def\Ker{\mathop{\rm Ker}\nolimits}

\def\Im{\mathop{\rm Im}\nolimits}
\def\dim{\mathop{\rm dim_k}\nolimits}

\newenvironment{proof}{\noindent {\it\bf Proof.} \rm}

\def \Sub {\mathop{\rm Sub} \nolimits}
\def\Hom{\mathop{\rm Hom}\nolimits}
\def\Ext{\mathop{\rm Ext}\nolimits}
\def\Ker{\mathop{\rm Ker}\nolimits}
\def\dim{\mathop{\rm dim_k}\nolimits}
\def\HH {\mathop{\rm HH}\nolimits}

\def\HH {\mathop{\rm HH}\nolimits}

\def\R{\mathcal R}
\def\P{\mathcal P}

\begin{document}

\title{Hochschild cohomology of triangular string algebras and its ring structure}
\date{}
\author{Mar\'\i a Julia Redondo and Lucrecia Rom\'an \footnote{Instituto de Matem\'atica,
Universidad Nacional del Sur, Av. Alem 1253, (8000) Bah\'\i a Blanca, Argentina.
{\it E-mail address:  mredondo@criba.edu.ar, lroman@uns.edu.ar}}
\thanks{The first author is a researcher from CONICET, Argentina. This work has been supported by the project PICT-2011-1510.  }}

\maketitle

\begin{abstract}
We compute the Hochschild cohomology groups $\HH^*(A)$ in case $A$ is a triangular string algebra, and show that its ring structure is trivial.
\end{abstract}

\section{Introduction}

Let $A$ be an associative, finite dimensional algebra over an algebraically closed field $k$.  It is well known that there exists a finite quiver $Q$ such that $A$ is Morita equivalent to $kQ/I$, where $kQ$ is the path algebra of $Q$ and $I$ is an admissible two-sided ideal of $kQ$.  

A finite dimensional algebra is called {\it biserial} if the radical of every projective indecomposable module is the sum of two uniserial modules whose intersection is simple or zero, see~\cite{F}.  
These algebras have been studied by several authors and from different points of view since there are a lot of natural examples of algebras which turn out to be of this kind. 

The representation theory of these algebras was first studied by Gelfand and Ponomarev in~\cite{GP}: they have provided the methods in order to classify all their indecomposable representations. This classification shows that biserial algebras are always tame, see also~\cite{WW}. They are an important class of algebras whose representation theory has been very well described, see~\cite{AS,BR}. 

The subclass of {\it special biserial algebras} was first studied by Skowro\'nski and Waschb{\"u}ch in~\cite{SW} where they  characterize
the biserial algebras of finite representation type.  A classification of the special biserial algebras which are minimal representation-infinite has been given by Ringel in \cite{R}.

An algebra is called a  {\it string algebra} if it is Morita equivalent to a monomial special biserial algebra.

The purpose of this paper is to study the Hochschild cohomology groups of a string algebra $A$ and describe its ring structure.

Since $A$ is an algebra over a field $k$, the Hochschild cohomology goups ${\HH}^i(A,M)$ with coefficients in an $A$-bimodule $M$ can be identified with the groups $\Ext_{A-A}^i(A,M)$.  In particular, if $M$ is the $A$-bimodule $A$, we simple write ${\HH}^i(A)$.

Even though the computation of the Hochschild cohomology groups ${\HH}^i(A)$ is rather complicated, some approaches have been successful when the algebra $A$ is
given by a quiver with relations. For instance, explicit formula for the dimensions of ${\HH}^i(A)$ in terms of
those combinatorial data have been found in~\cite{blm,c1,c2,crs,H,Re}. In particular, Hochschild cohomology of special biserial algebras has been considered in~\cite{Bu,ST}.

In the particular case of monomial algebras, that is, algebras $A=KQ/I$ where $I$ can be choosen as generated by paths, one has a detalied description of a minimal resolution of the $A$-bimodule $A$, see~\cite{B}. In general, the computation of the Hochschild cohomology groups using this resolution may lead to hard combinatoric computations.  However, for string algebras the  resolution, and the complex associated, are easier to handle.

The paper is organized as follows.  In Section 2 we introduce all the necessary terminology.  In Section 3 we recall the resolution given by Bardzell for monomial algebras in \cite{B}.  In Section 4 we present all the computations that lead us to Theorem \ref{HH} where we present the dimension of all the Hochschild cohomology groups of triangular string algebras. In Section 5 we describe the ring structure of the Hochschild cohomology of triangular string algebras.

\section {Preliminaries}

\subsection{Quivers and relations}

Let $Q$ be a  finite quiver with a set of vertices $Q_0$, a set of arrows $Q_1$ and $s, t : Q_1 \to Q_0$ be the maps associating to each arrow $\alpha$ its source  $s(\alpha)$ and its target $t(\alpha)$.  A path $w$ of length $l$ is a sequence of $l$ arrows $\alpha_1 \dots \alpha_l$ such that $t(\alpha_i)=s(\alpha_{i+1})$. We denote by $\vert w \vert $ the length of the path $w$. We put $s(w)=s(\alpha_1)$ and $t(w)=t(\alpha_l)$. For any vertex $x$ we consider $e_x$ the trivial path of length zero and we put $s(e_x)=t(e_x)=x$.  An oriented cycle is a non trivial path $w$ such that $s(w)=t(w)$.  If $Q$ has no oriented cycles, then $A$ is said a {\it triangular algebra}.

We say that a path $w$ divides a path $u$ if $u = L(w) w R(w)$, where $L(w)$ and $R(w)$ are not simultaneously paths of length zero.

The path algebra $kQ$ is the $k$-vector space with basis the set of paths in $Q$; the product on the basis elements is given by the concatenation of the sequences of arrows of the paths $w$ and $w'$ if they form a path (namely, if $t(w)=s(w')$) and zero otherwise.  Vertices form a complete set of orthogonal idempotents. Let $F$ be the two-sided ideal of $kQ$ generated by the arrows of $Q$. A two-sided ideal $I$ is said to be \textit{admissible} if there exists an integer $m \geq 2$ such that $F^m \subseteq I \subseteq F^2$.  The pair $(Q,I)$  is called a \textit{bound quiver}.

It is well known that if $A$ is a basic, connected, finite dimensional algebra over an algebraically closed field $k$, then there exists a unique finite quiver $Q$ and a surjective morphism of $k$-algebras $\nu: kQ \to A$, which is not unique in general, with $I_\nu=\Ker \nu$ admissible.  The pair $(Q,I_\nu)$  is called a \textit{presentation} of $A$.  The elements in $I$ are called {\it relations}, $kQ/I$ is said a {\it monomial algebra} if the ideal $I$ is generated by paths, and a relation is called {\it quadratic} if it is a path of length two.

\subsection{String algebras}

Recall from~\cite{SW} that a bound quiver $(Q,I)$ is  {\it special biserial} if it satisfies the following conditions:
\begin{itemize}
\item [S1)] Each vertex in $Q$ is the source of at most two arrows and the target of at most two arrows;
\item [S2)] For an arrow $\alpha$ in $Q$ there is at most one arrow $\beta$ and at most one arrow $\gamma$ such that $\alpha\beta \not \in I$ and $ \gamma \alpha \not \in I$.
\end{itemize}
If the ideal $I$ is generated by paths, the bound quiver $(Q,I)$ is  {\it string}.

\medskip

An algebra is called {\it special biserial} (or {\it string}) if it is Morita equivalent to a path algebra $kQ/I$ with $(Q,I)$ a special biserial bound quiver (or a string bound quiver, respectively).

Since Hochschild cohomology is invariant under Morita equivalence, whenever we deal with a string algebra $A$ we will assume that it is given by a string presentation $A = kQ/I$ with $I$ satisfying the previous conditions.  We also assume that the ideal $I$ is generated by paths of minimal length, and we fix a minimal set $\R$ of paths, of minimal length, that generate the ideal $I$.  Moreover, we fix a set $\P$ of paths in $Q$ such that the set $\{ \gamma + I, \gamma \in \P\} $ is a basis of $A=kQ/I$.

\section{Bardzell's resolution}

We recall that the Hochschild cohomology groups ${\HH}^i(A)$ of an algebra $A$ are the groups $\Ext_{A-A}^i(A,A)$.  Since string algebras are monomial algebras, their Hochschild cohomology groups can be computed using a convenient minimal projective resolution of $A$ as $A$-bimodule given in~\cite {B}.

In order to describe this minimal resolution, we need some definitions and notations.

Recall that we have fix a minimal set $\R$ of paths, of minimal length, that generate the ideal $I$.  It is clear that no divisor of an element in $\R$ can belong to $\R$.

The $n$-concatenations are elements defined inductively as follows: given any directed path $T$ in $Q$, consider the set of vertices that are starting and ending points of arrows belonging to $T$, and consider the natural order $<$ in this set.  Let $\R(T)$ be the set of paths in $\R$ that are contained in the directed path $T$. Take $p_1 \in \R(T)$ and consider the set $$L_1=\{ p \in \R(T): s(p_1)  < s(p) < t(p_1)  \}.$$
If $L_1 \not = \emptyset$, let $p_2$ be such that $s(p_2)$ is minimal with respect to all $p \in L_1$. Now assume that $p_1, p_2, \dots, p_j$ have been constructed.  Let
$$L_{j+1}=\{ p \in \R(T): t(p_{j-1}) \leq s(p) < t(p_j) \}.$$
If $L_{j+1} \not = \emptyset$, let $p_{j+1}$ be such that $s(p_{j+1})$ is minimal with respect to all $p \in L_{j+1}$. Thus $(p_1, \dots, p_{n-1})$ is an $n$-concatenation and we denote by $w(p_1, \dots, p_{n-1})$ the path from $s(p_1)$ to $t(p_{n-1})$ along the directed path $T$, and we call it the {\it support} of the concatenation.

These concatenations can be pictured as follows:
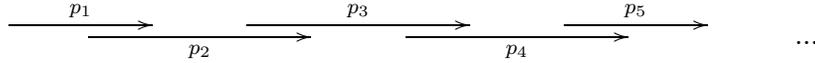
\[ \xymatrix{ \ar@<1ex>[rr]^{p_1} & \ar[rrr]_{p_2} &  & \ar@<1ex>[rrr]^{p_3} & & \ar[rrr]_{p_4} & & \ar@<1ex>[rr]^{p_5} & & & ...  } \]
Let $AP_0=Q_0$, $AP_1=Q_1$ and $AP_n$ the set of supports of $n$-concatenations.

\medskip

The construction of the sets $AP_{n}$ can also be done dually.   Given any directed path $T$ in $Q$ take $q_1 \in \R(T)$ and consider the set $$L_1^{op} =\{ q \in \R(T): s(q_1)  < t(q) < t(q_1)  \}.$$
If $L_1^{op} \not = \emptyset$, let $q_2$ be such that $t(q_2)$ is maximal with respect to all $q \in L_1^{op}$. Now assume that $q_1, q_2, \dots, q_j$ have been constructed.  Let
$$L_{j+1}^{op}=\{ q \in \R(T): s(q_{j}) < t(q) \leq s(q_{j-1}) \}.$$
If $L_{j+1}^{op} \not = \emptyset$, let $q_{j+1}$ be such that $t(q_{j+1})$ is maximal with respect to all $q \in L_{j+1}^{op}$.
Thus $(q_{n-1} , \dots, q_1)$ is an $n$-op-concatenation, we denote by $w^{op}(q_{n-1}, \dots, q_1)$ the path from $s(q_{n-1})$ to $t(q_1)$ along the directed path $T$,  we call it the support of the concatenation and we denote by $AP^{op}_{n}$ the set of supports of $n$-op-concatenations constructed in this dual way. Moreover, we denote $w^{op}(q_{n-1}, \dots, q_1)=w^{op}(q^1, \dots, q^{n-1})$.

\medskip

For any $w \in AP_n$ define $\Sub(w)= \{ w' \in AP_{n-1}: w' \ \mbox{divides} \ w\}$.

\begin{example}
Consider the following relations contained in a directed path $T$:
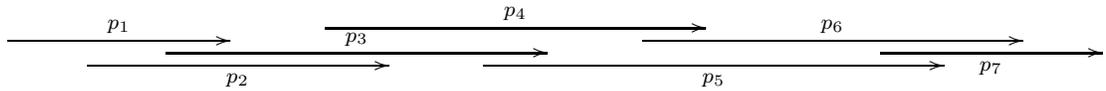
\[ \xymatrix{ \ar@<1ex>[rrr]^{p_1} & \ar@<-1ex>[rrrr]_{p_2} &  \ar[rrrrr]^{p_3} &
& \ar@<2ex>[rrrrr]^{p_4} 
& & \ar@<-1ex>[rrrrrr]_{p_5} & & \ar@<1ex>[rrrrr]^{p_6} & & & \ar[rrr]_{p_7} &&&&
 } \]
Then $w=w(p_1, p_2, p_4, p_5, p_7)$ is a $6$-concatenation,  $w=w^{op}(p_1, p_3, p_4, p_6, p_7)$ and
$$\Sub(w) = \{ w(p_1, p_2, p_4, p_5), w(p_2, p_3, p_5, p_6), w(p_3, p_4, p_6, p_7) \}.$$
\end{example}

\begin{lemma} \cite[Lemma 3.1]{B} \label{opuesto}
If $n \geq 2$ then $AP_{n} = AP^{op}_{n}$.
\end{lemma}

\begin{proof}
We will prove that for any $n$-concatenation $(p_1, \dots, p_{n-1})$ there exists a unique $n$-op-concatenation $(q^1, \dots, q^{n-1})$ such that $w(p_1, \dots, p_{n-1}) = w^{op} (q^1, \dots, q^{n-1})$.  The converse statement can be proved similarly. 
First observe that $w(p_1)=w^{op}(p_1)$ and $w(p_1, p_2)=w^{op}(p_1, p_2)$. Assume that $n >3$.  It is clear that $q^{n-1}=p_{n-1}$ since they are relations in $\R$ contained in the same path and sharing target.  When we look for $q^{n-2}$ we can observe that the maximality of its target implies that $t(p_{n-2}) \leq t(q^{n-2})$.  Since elements in $\R$ are paths of minimal length, $s(p_{n-2}) \leq s(q^{n-2})$. Now $t(q^{n-2}) < t(q^{n-1}) = t(p_{n-1})$ says that $q^{n-2} \not = p_{n-1}$ and the minimality of the starting point of $p_{n-1}$ says that $s(q^{n-2}) < t(p_{n-3})$.
Then
\[s(p_{n-2} )\leq  s(q^{n-2} )< t(p_{n-3} )\quad \mbox{and} \quad t(p_{n-2}) \leq t(q^{n-2}) <  t(p_{n-1}).\]
Now we describe $q^{n-3}$:  since $s(q^{n-2}) < t(p_{n-3} ) \leq s(p_{n-1})= s(q^{n-1})$ the maximality of the target of $q^{n-3}$ implies that $t(p_{n-3}) \leq t(q^{n-3})$.  Since elements in $\R$ are paths of minimal length, $s(p_{n-3}) \leq s(q^{n-3})$. Now $t(q^{n-3}) \leq s(q^{n-1}) = s(p_{n-1}) <  t(p_{n-2})$ says that $q^{n-3} \not = p_{n-2}$ and the minimality of the starting point of $p_{n-2}$ says that $s(q^{n-3}) < t(p_{n-4})$.
Then
\[ s(p_{n-3} )\leq  s(q^{n-3} )< t(p_{n-4} )\quad \mbox{and} \quad t(p_{n-3}) \leq t(q^{n-3}) <  t(p_{n-2}).\]
Since $s(q^{n-3}) < t(p_{n-4} ) \leq  s(p_{n-2}) \leq s(q^{n-2})$ we can continue this procedure in order to prove that, for $j=2, 3, \dots, n-2$, the element $q^{n-j}$  is such that
\[s(p_{n-j} )\leq  s(q^{n-j} )< t(p_{n-j-1} ), \quad  t(p_{n-j}) \leq t(q^{n-j}) <  t(p_{n-j+1})\]
and
\[s(q^{n-j}) < t(p_{n-j-1} ) \leq  s(p_{n-j+1}) \leq s(q^{n-j+1}).\]
Finally  the minimality of the source of $p_2$ and the inequality $t(p_1) \leq t(q^1) <  t(p_2)$ shows that $q^1=p_1$. \qed
\end{proof}

\begin{lemma}\label{particion}
If $n, m \geq 0$, $n + m \geq 2$  then any $w(p_1, \dots, p_{n+m-1})  \in AP_{n+m}$ can be written in a unique way as
$$w(p_1, \dots, p_{n+m-1}) = {^{(n)}\! w} \ u \ w^{(m)}$$
with $ {^{(n)}\! w} = w(p_1, \dots, p_{n-1}) \in AP_{n}$, $w^{(m)}= w^{op}(q^{n+1}, \dots, q^{n+m-1}) \in AP^{op}_{m}$ and $u$ a path in $Q$. Moreover, $p_n = a \ u \ b$ and $q^n= a' \ u \ b'$ with $a, a', b, b'$ non trivial paths, and hence $u \in \P$.
\end{lemma}

\begin{proof}     
From Lemma \ref{opuesto} we know that $w(p_1, \dots, p_{n+m-1}) = w^{op}(q^1, \dots, q^{n+m-1})$. It is clear that  $w(p_1, \dots, p_{n-1}) \in AP_{n}$ and $w^{op}(q^{n+1}, \dots, q^{n+m-1}) \in AP^{op}_{m}$.  In order to prove the existence of a path $u$ we just have to observe that the proof of the previous lemma and the definition of concatenations imply that
\[t(p_{n -1}) \leq t(q^{n -1}) \leq s(q^{n +1}).\]
Finally, the relation of $u$ with $p_n$ and $q^n$ follows from the inequalities
\[s(p_n) < t(p_{n-1}) \leq s(q^{n+1}) < t(p_n)\]
and
\[s(q^n) < t(p_{n-1}) \leq s(q^{n+1}) < t(q^n).\] \qed
\end{proof}

Now we want to study the sets $\Sub(w)$ in some particular cases. Observe that for any $w \in AP_n$, $\psi_1 = w^{op}(q^2, \cdots, q^{n-1})$ and $\psi_2 = w(p_1, \cdots, p_{n-2})$  belong to $\Sub(w)$ and $w= L(\psi_1) \psi_1 = \psi_2 R(\psi_2)$.

\begin{lemma}\label{cuadratica}
If $w=w(p_1, \dots, p_{n-1}) \in AP_{n}$ is such that $p_i$ has length two for some $i$ with $1 \leq i \leq n-1$, then $\vert \Sub (w) \vert =2$.
\end{lemma}

\begin{proof}
Assume that $p_i= \alpha \beta$. If $i=1$, then any $n-1$-concatenation different from $(p_1, \dots, p_{n-2})$ and corresponding to an element in $\Sub(w)$ must correspond to a divisor of $w(p_2, \dots, p_{n-1})$, hence it is equal to $(p_2, \dots, p_{n-1})$.  The proof for $i=n-1$ is similar.
If $1 < i < n-1$ and $\hat w \in \Sub(w)$ then $\hat w$ also contains the quadratic relation $p_i$ and by the previous lemma we have that 
\[w = {^{(i)}\! w}  \ w^{(n-i)}, \ \hat w = {^{(j)}\! \hat w}  \ \hat  w^{(n-1-j)}\]
with $t({^{(i)}\! w}) = t( {^{(j)}\! \hat w})$ and $s( w^{(n-i)}) = s(\hat  w^{(n-1-j)})$.  Then $\hat w$ is ${^{(i)}\! \hat w}  \ \hat  w_1$ or $\hat w_2  \ \hat  w^{(n-i)}$, where $\hat w_1$ is the unique element in $\Sub(  w^{(n-i)})$ sharing source with $w^{(n-i)}$ and  $\hat w_2$ is the unique element in $\Sub({^{(i)}\! w})$ sharing target with ${^{(i)}\! w}$. \qed
\end{proof}

\begin{lemma}\label{p=q}
If $w=w(p_1, \cdots, p_{n-1}) = w^{op}(q^1, \cdots, q^{n-1})$ and $q^{m}$ has length two for some $m$ such that $1 < m < n$ then $q^m=p_m$ and $q^{m-1} = p_{m-1}$.
\end{lemma}

\begin{proof}
Let $q^m= \alpha  \beta $.  In the proof of Lemma \ref{opuesto} we have seen that
\[  s(q^{m} )< t(p_{m-1}) \leq t(q^{m-1}) .\]
Now $s(q^{m} )= s(\alpha)$  and $t(q^{m-1}) = t(\alpha)$, so $t(p_{m-1})  = t(q^{m-1})$ and hence 
$p_{m-1}=q^{m-1}$.  Analogously,
\[  s(q^{m+1} )< t(p_{m}) \leq t(q^{m}) , \]   
$s(q^{m+1}) = s(\beta)$  and $t(q^{m}) = t(\beta)$, so $t(p_{m})   = t(q^{m})$ and hence 
$p_{m}=q^{m}$. \qed
\end{proof}

\medskip

In some results that will be shown in the following sections, we will need a description of right divisors of paths of the form $w u$, for $w$ the support of a concatenation and $u \in \P$.  Their existence depends on each particular case as we show in the following example.

\begin{example}
Let   $w=w(p_1, p_2, p_3, p_4) \in AP_5$, $u \in \P$ with $t(w)=s(u)$.  Observe that the existence of a divisor $\psi \in AP_{n}$, for $n=4,5$ such that 
$w u = L(\psi) \psi$ depends on the existence of appropiate relations.  For instance, if $wu$ is the following path
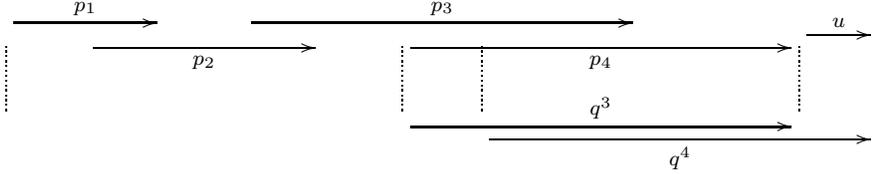
\[ \xymatrix{ 
\ar@<1ex>[rr]^{p_1} \ar@{.}[d] & \ar@<-1ex>[rrr]_{p_2}  & & \ar@<1ex>[rrrrr]^{p_3} &    & \ar@{.}[d] \ar@<-1ex>[rrrrr]_{p_4} &  \ar@{.}[d] &  & & & \ar@{.}[d] \ar[r]^{u} &  \ar@{.}[d] & & &\\
 & & & &  & \ar@<-1ex>[rrrrr]^{q^3} &  \ar@<-2ex>[rrrrr]_{q^4} & & & & & & } \]
the existence of $\psi = \psi^{op}(q^1, q^2, q^3, q^4)$ depends on the existence of a relation whose ending point is between $s(q^3 )$ and $s(q^4)$.
\end{example}

Part of the following lemma has also been  proved in \cite[Lemma 3.2]{B}.

\begin{lemma} \label{A}
If $n$ is even let $w=w(p_1, \cdots, p_{n-1}) \in AP_n$.
\begin{itemize}
\item [(i)] If $v=v^{op}(q^2, \cdots, q^{n-1}) \in AP_{n-1}$ is such that $w \ a = b \ v  \not \in AP_{n+1}$ with $a, b$ paths  in $Q$, $a \in \P$ then $t(p_1) \leq s(q_2)$, and 
\item [(ii)] If $u=u^{op}(q^1, \cdots, q^{n-1}) \in AP_{n}$ is such that $w \ a = b \ u \not \in AP_{n+1}$ with $a, b$ paths  in $\P$ then there exists $z \in AP_{n+1}$ such that $z$ divides the path $T$ that contains $w$ and $v$ and $t(z)=t(u)$.
\end{itemize}
\end{lemma}

\begin{proof}
\begin{itemize}
\item [(i)] The assumption  $a \in \P$ implies that
\[ s(p_{n-1}) < s(q^{n-1}) < t(p_{n-1})\]
and moreover
\[ s(p_{n-1}) < s(q^{n-1}) < t(p_{n-2})\] since otherwise
$w \ a  \in AP_{n+1}$ because $q^{n-1}$ would belong to the set considered in order to choose $p_n$.
Now \[q^{n-2} = p_{n-1} \ \mbox{or} \ t(p_{n-1}) < t(q^{n-2}) \ \mbox{and hence} \ s(p_{n-1}) < s(q^{n-2}) < s(q^{n-1}).\]
Now $q^{n-3}$ is such that $s(p_{n-3}) < s(q^{n-3}) < s(q^{n-1})$. The minimality of $s(p_{n-2})$ says that 
\[  s(p_{n-3}) < s(q^{n-3}) < t(p_{n-4}) .\]
An inductive procedure shows that 
\[ s(p_{n-2j+1}) < s(q^{n-2j+1}) < t(p_{n-2j})\]
and 
\[q^{n-2j} = p_{n-2j+1} \ \mbox{or} \ t(p_{n-2j+1}) < t(q^{n-2j}) \ \mbox{and hence} \ s(p_{n-2j+1}) < s(q^{n-2j}) < s(q^{n-2j+1})\]
for any $j$ such that $1 \leq 2j-1, 2j \leq n-1$.  In particular, since $n$ is even we have that
\[q^2 = p_3 \ \mbox{or} \ s(p_3) < s(q^2) < s(q^3)\]
and hence $t(p_1) \leq s(q^2)$.

\item[(ii)]  In  order to prove the existence of $z$ we have to show that there exists $q^0 \in \R(T)$ such that $z=z^{op}(q^0, q^1, \cdots, q^{n-1})$ belongs to $AP_{n+1}$, that is, we have to see that the set 
$\{ q \in \R(T) : s(q^1) < t(q) \leq s(q^2) \}$
is not empty.  Suppose it is empty.  The assumption  $b \in \P$ implies that $s(q^2) < t(p_1)$, a contradiction from (i).
\end{itemize} \qed 
\end{proof}

\begin{lemma}\cite[Lemma 3.3]{B}  \label{dos}
If $m \geq 1$ and $w \in AP_{2m+1}$ then $\vert \Sub(w)\vert =2$.
\end{lemma}

Now we are ready to describe the minimal resolution constructed by Bardzell in \cite{B}:
$$ \dots \longrightarrow A \otimes k AP_n \otimes A \stackrel{d_{n}}{\longrightarrow} A \otimes k AP_{n-1} \otimes A \longrightarrow \dots \longrightarrow  A \otimes k AP_0 \otimes A \stackrel{\mu}{\longrightarrow} A \longrightarrow 0$$
where $kAP_0=kQ_0$, $kAP_1=kQ_1$ and $kAP_{n}$ is the vector space generated by the set of supports of $n$-concatenations and all tensor products are taken over $E=kQ_0$, the subalgebra of $A$ generated by the vertices.  

In order to define the $A$-$A$-maps
\[d_n: A \otimes k AP_{n} \otimes A \to A \otimes k AP_{n-1} \otimes A \] we need the following notations: if $m \geq 1$, for any $w \in AP_{2m+1}$ we have that $\Sub(w) = \{ \psi_1, \psi_2 \}$ where $w= L(\psi_1) \psi_1 = \psi_ 2 R(\psi_2)$; and for any $w \in AP_{2m}$ and $\psi \in \Sub(w)$ we denote $w= L(\psi) \psi R(\psi)$. Then

\begin{eqnarray*}
\mu( 1 \otimes e_i \otimes 1 ) & = & e_i \\
d_1( 1 \otimes \alpha \otimes 1) & = &  \alpha \otimes e_{t(\alpha)} \otimes 1 - 1 \otimes e_{s(\alpha)} \otimes \alpha \\
d_{2m} (1 \otimes w \otimes 1) & = & \sum_{\psi \in \Sub(w)} L(\psi) \otimes \psi \otimes R(\psi) \\
d_{2m+1} (1 \otimes w \otimes 1) & = & L(\psi_1) \otimes \psi_1 \otimes 1 - 1 \otimes \psi_2 \otimes R(\psi_2)
\end{eqnarray*}
The $E-A$ bilinear map $c: A \otimes k AP_{n-1} \otimes A \to A \otimes k AP_{n} \otimes A$ defined by
\[ c( a \otimes \psi \otimes 1) = \sum_{\substack{ w \in AP_n \\ L(w) w R(w) =  a \psi}}  L(w) \otimes w \otimes R(w)\]
is a contracting homotopy, see \cite[Theorem 1]{S} for more details.

\section{Computations}

The Hochschild complex, obtained by applying $\Hom_{A-A} (-, A)$ to the Hochschild resolution we described in the previous section and using the isomorphisms

\[ \Hom_{A-A}( A \otimes k AP_n \otimes ,  A)  \simeq \Hom_{E-E}( k AP_n,  A)\]
is
\[ 0 \longrightarrow \Hom_{E-E} (k AP_0 , A)  \stackrel{F_{1}}{\longrightarrow} \Hom_{E-E}( k AP_1,  A ) \stackrel{F_{2}}{\longrightarrow}   \Hom_{E-E}(  k AP_2, A)   \cdots \]
where

\begin{eqnarray*}
F_1( f) (\alpha)  & = &  \alpha f(e_{t(\alpha)}) - f(e_{s(\alpha)}) \alpha\\
F_{2m}(f) (w)  & = & \sum_{\psi \in \Sub(w)} L(\psi) f(\psi)R(\psi) \\
F_{2m+1} (f) (w)  & = & L(\psi_1) f(\psi_1) - f(\psi_2) R(\psi_2). \\
\end{eqnarray*}
In order to compute its cohomology, we need a handle description of this complex: we will describe explicite basis of these $k$-vector spaces and study the behaviour of the maps between them in order to get information about kernels and images.

\medskip

Recall that we have fixed a set $\P$ of paths in $Q$ such that the set $\{ \gamma + I: \gamma \in \P\} $ is a basis of $A=kQ/I$.
For any subset $X$ of paths in $Q$, we denote $(X// \P)$ the set of
pairs $(\rho, \gamma) \in X \times \P$ such that $\rho, \gamma$ are
parallel paths in $Q$, that is
\[ (X// \P) = \{  (\rho, \gamma) \in X \times \P: s(\rho) = s(\gamma),   t(\rho) =
t(\gamma)\}.\]
Observe that the  $k$-vector spaces  $\Hom_{E-E} (k AP_{n} , A)$ and $k(AP_{n}// \P)$ are isomorphic, and from now on we will identify elements $(\rho, \gamma) \in (AP_{n}// \P)$ with basis elements $f_{(\rho, \gamma)}$ in $\Hom_{E-E} (k AP_{n} , A)$ defined by
\[ f_{(\rho, \gamma)} (w)= \begin{cases}
\gamma & \mbox{if $w=\rho$}, \\
0 & \mbox{otherwise}.
\end{cases}\]
Now we will introduce several subsets of $(AP_{n}// \P)$  in order to get a nice description of the kernel and the image of $F_n$.
For $n=0$ we have that $(AP_0//\P) = (Q_0, Q_0)$.  For $n=1$, $(AP_1// \P) = (Q_1//\P)$, and we consider the following partition
\[(Q_1//\P) = (1,1)_1 \cup (0,0)_1\]
where
\begin{eqnarray*}
(1,1)_1 &=& \{ (\alpha, \alpha) : \alpha \in Q_1 \} \\
(0,0)_1 &=& \{ (\alpha, \gamma)  \in   (Q_1//\P) :  \alpha \not = \gamma \}\\
\end{eqnarray*}
For any $n \geq 2$ let
\begin{eqnarray*}
(0,0)_n &=& \{ (\rho, \gamma) \in (AP_{n} // \P) : \rho= \alpha_1 \hat {\rho} \alpha_2 \ \mbox{and} \ \gamma \not \in \alpha_1 kQ \cup kQ \alpha_2\} \\
 (1,0)_n &=& \{ (\rho, \gamma) \in (AP_{n} // \P) : \rho= \alpha_1 \hat {\rho} \alpha_2 \ \mbox{and} \ \gamma  \in \alpha_1 kQ, \gamma \not \in  kQ \alpha_2\} \\
 (0,1)_n &=& \{ (\rho, \gamma) \in (AP_{n} // \P) : \rho= \alpha_1 \hat {\rho} \alpha_2 \ \mbox{and} \ \gamma \not \in \alpha_1 kQ, \gamma \in  kQ \alpha_2\} \\
(1,1)_n &=& \{ (\rho, \gamma) \in (AP_{n} // \P) : \rho= \alpha_1 \hat {\rho} \alpha_2 \ \mbox{and} \ \gamma \in \alpha_1 kQ \alpha_2\}
\end{eqnarray*}

\bigskip

\begin{remark}\label{uno}
\begin{itemize}
\item [ ]
\item[1)]
These subsets are a partition of $(AP_{n}// \P)$.
\item[2)] Any $(\rho, \gamma) \in (AP_{n}// \P)$ verifies that $\rho$ and $\gamma$ have at most one common first arrow and at most one common last arrow: if $\rho= \alpha_1 \dots \alpha_s \beta \overline{\rho}$ and $\gamma=\alpha_1 \dots \alpha_s \delta \overline{\gamma}$, with $\beta, \delta$ different arrows, then $\alpha_s \beta \in I$.  Since $\alpha_1 \dots \alpha_s$ is a factor of $\gamma$, and $\gamma$ belongs to $\P$, then $\alpha_1 \dots \alpha_s \not \in I$.  But the $n$-concatenation associated to $\rho$ must start with a relation in $\R(T)$, so $s=1$,  this concatenation starts with the relation $\alpha_s \beta$ and the second relation of this concatenation starts in $s(\beta)$.
\item[3)] If $(\rho, \gamma) \in  (1,0)_n$, $\rho = \alpha_1 \overline{\rho}, \gamma = \alpha_1 \overline{\gamma}$ then $\overline{\rho} \in AP_{n-1}$ and $(\overline{\rho},  \overline{\gamma}) \in (0,0)_{n-1}$.  The same construction holds in $(0,1)_n$. Finally, if $(\rho, \gamma) \in  (1,1)_n$, $\rho= \alpha_1 \hat {\rho} \alpha_2, \gamma= \alpha_1 \hat {\gamma} \alpha_2$ then $\hat {\rho} \in AP_{n-2}$ and  $(\hat {\rho},\hat {\gamma}) \in (0,0)_{n-2}$.
\item[4)]  If $(\rho, \gamma) \in (AP_{2}// \P)= (\R // \P)$, we have already seen that $\rho$ and $\gamma$ have at most one common first arrow and at most one common last arrow.  Assume that $\rho= \alpha_1 \alpha_2 \overline \rho$, $\gamma = \alpha_1 \beta \overline \gamma$.  Since $A$ is a string algebra and $\gamma \not \in I$ we have that $\alpha_1 \alpha_2 \in I$ and hence $\rho=\alpha_1 \alpha_2$.  Since we are dealing with triangular algebras, we also have that $\rho$ and $\gamma$ can not have simultaneously one common first arrow and one common last arrow. Then
\begin{eqnarray*}
(1,1)_2 &=& \emptyset \\
(1,0)_2 &=& \{ (\rho, \gamma) \in (\R, \P)  : \rho= \alpha_1  \alpha_2,  \gamma \in \alpha_1  kQ, \gamma \not \in kQ \alpha_2 \} \\
(0,1)_2 &=& \{ (\rho, \gamma) \in (\R, \P)  : \rho= \alpha_1  \alpha_2,  \gamma \not  \in \alpha_1  kQ, \gamma \in kQ \alpha_2 \}.
\end{eqnarray*}
\end{itemize}
\end{remark}
We also have to distinguish elements inside each of the previous sets taking into account the following definitions:

\begin{eqnarray*}
 ^+(X // \P) & = & \{ (\rho, \gamma) \in (X // \P) : Q_1 \gamma \not \subset I  \} \\
 ^-(X // \P) & = & \{ (\rho, \gamma) \in (X // \P) : Q_1 \gamma  \subset I \}
\end{eqnarray*}
In an analogous way we define $(X// \P)^+$,  $(X// \P)^-$,   $^+(X// \P)^+=  ^+(X// \P) \cap  (X// \P)^+$ and so on. \\
Finally we define
\begin{eqnarray*}
(1,0)_n^{--} &=& \{ (\rho, \gamma) \in (1,0)_n^{-}  : \rho= \alpha_1 \hat {\rho} \alpha_2,  \gamma = \alpha_1 \hat {\gamma} , \hat{\gamma} Q_1 \subset I \} \\
(1,0)_n^{-+} &=& \{ (\rho, \gamma) \in (1,0)_n^{-} : \rho= \alpha_1 \hat {\rho} \alpha_2,  \gamma = \alpha_1 \hat {\gamma} , \hat{\gamma} Q_1 \not \subset I \} \\
^{--}(0,1)_n &=& \{ (\rho, \gamma) \in \! ^-(0,1)_n : \rho= \alpha_1 \hat {\rho} \alpha_2,  \gamma =  \hat {\gamma} \alpha_2 , Q_1 \hat{\gamma}  \subset I \} \\
^{+-}(0,1)_n &=& \{ (\rho, \gamma) \in \! ^{-}(0,1)_n : \rho=
\alpha_1 \hat {\rho} \alpha_2,  \gamma =  \hat {\gamma} \alpha_2,
Q_1 \hat{\gamma} \not \subset I \}
\end{eqnarray*}
Now we will describe the morphisms $F_n$ restricted to the subsets we have just defined.

\begin{lemma}\label{nucleo}
For any $n \geq 2$ we have
\begin{itemize}
\item[(a)] $^-(0,0)_{n-1}^- \cup (1,0)_{n-1}^- \cup \ ^-(0,1)_{n-1} \cup  (1,1)_{n-1} \subset \Ker F_n$;
\item[(b)] the function $F_n$ induces a bijection from $^-(0,0)_{n-1}^+$ to $^{--}(0,1)_n$;
\item[(c)] the function $F_n$ induces a bijection from $^+(0,0)_{n-1}^-$ to $(1,0)^{--}_n$;
\item[(d)] there exist bijections $\phi_m: (1,0)_{m}^+ \to {^+(0,1)_{m}}$ and $\psi_m: (1,0)_m^{-+} \to {^{+-}(0,1)_m}$ such that
$$(id +(-1)^{n-1} \phi_{n-1})((1,0)_{n-1}^+) \subset \Ker F_n,$$
$$(-1)^n F_n ((1,0)_{n-1}^+) = (1,1)_n$$
and $$F_n (^+(0,0)^+_{n-1})= (id  + (-1)^{n} \phi_n)((1,0)^+_n) \cup (id  + (-1)^{n} \psi_n)((1,0)^{-+}_n) .$$
\end{itemize}
\end{lemma}

\begin{proof}
\begin{itemize}
\item[(a)] In order to check that $(\rho, \gamma)$ belongs to $\Ker F_n$ we have to prove that for any $w \in AP_{n}$ such that $\rho$ divides $w$, that is, $w= L(\rho) \rho R(\rho)$ and $\vert L(\rho)\vert + \vert R(\rho) \vert > 0$, then
$L(\rho) \gamma R(\rho) \in I$. \\
If $(\rho, \gamma) \in {^-(0,0)_{n-1}^-}$  then $L(\rho) \gamma R(\rho) \in I$.\\
If $(\rho, \gamma) \in (1,0)_{n-1}^-$ then $\gamma R(\rho) \in I$ if $\vert R(\rho) \vert >0$.  On the other hand, if  $w= L(\rho) \rho$ we can deduce that $L(\rho) \gamma \in I$ using Remark \ref{uno} (2): if $L(\rho) \not \in I$ then the first relation in the $n$-concatenation corresponding to $w$ has $\alpha_1$ as it last arrow and $\gamma= \alpha_1 \hat \gamma$.\\
The proof for ${^-(0,1)_{n-1}}$ is analogous. \\
Finally, if $(\rho, \gamma) \in (1,1)_{n-1}$, the statement is clear for $n=2, 3$.  If $n>3$ and  $\rho = \alpha_1 \hat \rho \alpha_2$, from Remark \ref{uno} (2) we get that if $\vert L(\rho) \vert >0$ then the first relation in the  $n$-concatenation corresponding to $w$ has $\alpha_1$ as it last arrow, and if $\vert R(\rho) \vert >0$
then the last relation has $\alpha_2$ as it first arrow.  The assertion is clear since $\gamma= \alpha_1 \hat \gamma \alpha_2$ and hence $L(\rho) \gamma R(\rho)= L(\rho) \alpha_1 \hat \gamma \alpha_2 R(\rho) \in I$.
\item[(b)] If $(\rho, \gamma) \in {^-(0,0)_{n-1}^+}$ there exists a unique arrow $\beta$ such that $\gamma \beta \in \P$.  It is clear that $\rho \beta \in AP_{n}$, $(\rho \beta, \gamma \beta) \in \! {^{--}(0,1)_{n}}$ and $F_n (f_{(\rho, \gamma)})= (-1)^nf_{(\rho \beta, \gamma \beta)}$.
\item[(c)] Analogous to the previous one.
\item[(d)] If $(\alpha \hat \rho, \alpha \hat \gamma) \in (1,0)_m^+$ then there exists a unique arrow $\beta$ such $\alpha \hat \gamma \beta \in \P$.  It is clear that $\hat \rho \in AP_{m-1}$, $(\hat \rho, \hat \gamma) \in {^+(0,0)}_{m-1}^+$, 
$(\hat \rho \beta,\hat \gamma \beta) \in {^+(0,1)_{m}}$ and $(\alpha \hat \rho \beta, \alpha \hat \gamma \beta) \in {(1,1)_{m+1}}$.  The statement is clear if we define
$\phi_m (\alpha \hat \rho, \alpha \hat \gamma) = (\hat \rho \beta,\hat \gamma \beta)$ since
$$F_{m+1}(f_{(\hat \rho \beta,\hat \gamma \beta)}) = f_{(\alpha \hat \rho \beta, \alpha \hat \gamma \beta)} =(-1)^{m+1} F_{m+1} (f_{(\alpha \hat \rho, \alpha \hat \gamma)}).$$
In a similar way we can see that if $(\alpha \hat \rho, \alpha \hat
\gamma) \in (1,0)_m^{-+}$ then there exists a unique arrow $\beta$
such $\hat \gamma \beta \in \P$.  Now we have that $\hat \rho \in
AP_{m-1}$, $(\hat \rho,\hat \gamma ) \in{^+(0,0)_{m-1}^+}$ and $(\hat
\rho \beta,\hat \gamma \beta) \in {^{+-}(0,1)_{m}}$, so  it is enough
to define
$\psi_m (\alpha \hat \rho, \alpha \hat \gamma) = (\hat \rho \beta,\hat \gamma \beta)$.\\
Now if $(\hat \rho,\hat \gamma ) \in{^+(0,0)_{m-1}^+}$ there exist unique arrows $\alpha, \beta$ such that $\alpha \hat \gamma \in \P$ and $\hat \gamma \beta \in \P$.  If $\alpha \hat \gamma \beta \in \P$ then $(\alpha \hat \rho, \alpha \hat \gamma) \in (1,0)_m^+$ and $\phi_m (\alpha \hat \rho, \alpha \hat \gamma) = (\hat \rho \beta,\hat \gamma \beta) \in {^+(0,1)_{m}}$.
If $\alpha \hat \gamma \beta \in I$ then $(\alpha \hat \rho, \alpha \hat \gamma) \in (1,0)_m^{-+}$ and $\psi_m (\alpha \hat \rho, \alpha \hat \gamma) = (\hat \rho \beta,\hat \gamma \beta) \in {^{+-}(0,1)_{m}}$. 
In both cases
\[F_{m}(f_{(\hat \rho,\hat \gamma )}) = f_{(\alpha \hat \rho, \alpha \hat \gamma)} +(-1)^{m} f_{(\hat \rho \beta,\hat \gamma \beta)}.\]
\end{itemize} \qed
\end{proof}

\begin{lemma}
For any $n \geq 2$ we have that
\[\dim \Ker F_n = \vert  {^-(0,0)_{n-1}^-} \vert  +   \vert (1,0)_{n-1} \vert + \vert {^-(0,1)_{n-1}} \vert +  \vert (1,1)_{n-1} \vert \]
and
\[\dim \Im F_n = \vert {^{--}(0,1)_n} \vert + \vert (1,0)_n \vert + \vert (1,1)_n \vert  .\]
\end{lemma}

\begin{proof}
From the previous lemma we have that
\[\dim \Ker F_n = \vert {^-(0,0)_{n-1}^- } \vert + \vert (1,0)_{n-1}^- \vert + \vert  {^-(0,1)_{n-1}} \vert + \vert  (1,1)_{n-1} \vert + \vert  (1,0)_{n-1}^+ \vert  ,\]
but
\[ \vert (1,0)_{n-1} \vert  = \vert (1,0)_{n-1}^- \vert + \vert (1,0)_{n-1}^+ \vert.\]
Moreover
\[\dim \Im F_n = \vert  ^{--}(0,1)_n \vert + \vert (1,0)^{--}_n \vert  + \vert (1,1)_n \vert + \vert (1,0)_n^+ \vert + \vert (1,0)^{-+}_n) \vert  ,\]
but
\[  \vert (1,0)_{n} \vert  = \vert (1,0)_{n}^{--} \vert + \vert (1,0)_{n}^{-+} \vert  + \vert (1,0)_{n}^+ \vert .\] \qed
\end{proof}

\begin{teo}\label{HH}
If $A$ is a triangular string algebra, then
\[ \dim \HH^n (A) = \begin{cases}
1 & \mbox{if $n=0$} ,\\
\vert Q_1 \vert + \vert  {^-(0,0)^-_1} \vert - \vert Q_0 \vert + 1 &  \mbox{if $n=1$},\\
\vert ^{+-}(0,1)_n \vert + \vert {^-(0,0)^-_n} \vert  & \mbox{if $n \geq 2$.}
\end{cases} \]
\end{teo}

\begin{proof}
It is clear that $\HH^0 (A) = \Ker F_1$ is the center of $A$, and has dimension $1$ since $A$ is triangular. This implies that \[\dim \Im F_1 = \vert (Q_0 // Q_0) \vert - \dim \Ker F_1 = \vert Q_0 \vert - 1.\]
So
\[\dim \HH^1(A) = \dim \Ker F_2 - \vert Q_0 \vert + 1 = \vert (1,1)_1 \vert + \vert  {^-(0,0)^-_1} \vert - \vert Q_0 \vert + 1\]
since $(1,0)_1 = \emptyset = (0,1)_1$, and $\vert (1,1)_1 \vert = \vert Q_1 \vert$. 
 Finally for $n \geq 2$
\begin{eqnarray*}
\dim \HH^n(A) & =  & \dim \Ker F_{n+1} - \dim \Im F_n \\
& = & \vert (1,1)_{n} \vert +   \vert (1,0)_{n} \vert + \vert {^-(0,1)_{n}} \vert + \vert  {^-(0,0)_{n}^-} \vert - \vert (1,0)_n \vert - \vert (1,1)_n \vert - \vert {^{--}(0,1)_n}\vert \\
& = & \vert ^{+-}(0,1)_n \vert + \vert {^-(0,0)^-_n} \vert
\end{eqnarray*} \qed
\end{proof}

The following corollary includes the subclass of gente algebras, that is, string algebras $A=kQ/I$ such that $I$ is generated by quadratic relations and for any arrow $\alpha \in Q$ there is at most one arrow $\beta$ and at most one arrow $\gamma$ such that $\alpha \beta \in I$ and $\gamma \alpha \in I$.

\begin{coro}
If $A$ is a triangular quadratic algebra, then 
\[ \dim \HH^n (A) = \begin{cases}
1 & \mbox{if $n=0$}, \\
\vert Q_1 \vert + \vert  {^-(0,0)^-_1} \vert - \vert Q_0 \vert + 1 &  \mbox{if $n=1$},\\
\vert {^-(0,0)^-_n} \vert & \mbox{if $n \geq 2$.}
\end{cases} \]
\end{coro}

\begin{coro}\cite[Theorem 5.1]{ABL}
If $A$ is a triangular string algebra, the following conditions are equivalent:
\begin{itemize}
\item [i)] $\HH^1(A)=0$;
\item [ii)] the quiver $Q$ is a tree;
\item [iii)] $\HH^i(A)=0$ for $i>0$;
\item [iv)] $A$ is simply connected.
\end{itemize}
\end{coro}

\begin{proof}
It is well known that for monomial algebras, $A$ is simply connected if and only if $Q$ is a tree.
If $\HH^1(A)=0$, observing that $\vert  {^-(0,0)^-_1} \vert  \geq 0$ we have that $\vert Q_1 \vert  - \vert Q_0 \vert + 1 = 0$   .  Then the quiver $Q$ is a tree.  All the other implications are clear. \qed
\end{proof}

\begin{example}
Let $A_n= kQ/I$ with 
\[Q:  \xymatrix{ 
0 \ar@<1ex>[r]^{\alpha_1} \ar@<-1ex>[r]_{\beta_1} & 1 \ar@<1ex>[r]^{\alpha_2} \ar@<-1ex>[r]_{\beta_2} & 2 & \cdots & n-1 \ar@<1ex>[r]^{\alpha_n} \ar@<-1ex>[r]_{\beta_n} & n }\]
and $I = < \alpha_{i} \alpha_{i+1} , \beta_{i} \beta_{i+1} >_{\{ i=1, \cdots, n-1\} }$, $n \geq 1$.
Then
\[ dim \HH^i(A_1) = \begin{cases}
1 & \mbox{if $i=0$}, \\
3 & \mbox{if $i=1$},  \\
0 & \mbox{otherwise},\\
\end{cases} ,\qquad 
dim \HH^i(A_{2m}) = \begin{cases}
1 & \mbox{if $i=0$}, \\
2m & \mbox{if $i=1$}, \\
0 & \mbox{otherwise}\\
\end{cases}\]
and
\[
dim \HH^i(A_{2m+1}) = \begin{cases}
1 & \mbox{if $i=0$}, \\
2m+1 & \mbox{if $i=1$}, \\
2 & \mbox{if $i=2m+1$}, \\
0 & \mbox{otherwise}.\\
\end{cases}
\]
\end{example}

\section{Ring structure}

It is well known that the Hochschild cohomology goups $\HH^i(A)$ can be identified with the groups $\Ext_{A-A}^i(A,A)$, so the Yoneda product defines a product in the Hochschild cohomology $\sum_{i \geq 0} \HH^i(A)$ that coincides with the cup product as defined in \cite{G1,G2}.

Given $[f] \in \HH^m(A)$ and $[g] \in \HH^n(A)$, the cup product $[g \cup f] \in \HH^{n+m}(A)$ can be defined as follows: $g \cup f = g f_n$ where $f_n$ is a morphism making the following diagram commutative
\[ \xymatrix{ 
 A \otimes k AP_{m+n}\otimes A \ar[r]^-{d_{m+n}} \ar[d]^{f_n} &  A \otimes k AP_{m+n-1} \otimes A  \ar[r]^-{d_{m+n-1}}  \ar[d]^{f_{n-1}} & \dots \ar[r]^-{d_{m+1}}  & A \otimes k AP_m \otimes A \ar[d]^{f_0} \ar[rd]^f &  \\
  A \otimes k AP_{n} \otimes A  \ar[r]^-{d_{n}} \ar[d]^g &  A \otimes k AP_{n-1} \otimes A  \ar[r]^-{d_{n-1}}  & \dots   \ar[r] & A \otimes k AP_0  \otimes A \ar[r]^-\mu  & A \ar[r] & 0  \\
  A } \] 

In particular we are interested in maps $f \in  \Hom_{A-A}( A\otimes k AP_{m} \otimes A,  A )$ such that the associated morphism $\hat f \in  \Hom_{E-E}(  k AP_{m},  A )$ defined by $\hat f( w) = f (1 \otimes w \otimes 1)$ is in the kernel of the morphism $F_{m+1}:  \Hom_{E-E}( k AP_{m},  A )  \to  \Hom_{E-E}(  k AP_{m+1}, A) $ appearing in the Hochschild complex.

For any $m>0$ we will use Lemma \ref{particion} in order to the define maps $f_n$ that complete the previous diagram in a commutative way.  Recall that if $n>0$ any $w=w(p_1, \dots, p_{m+n-1})  \in AP_{m+n}$ can be written in a unique way as
\[w = {^{(n)}\! w} \ u \ w^{(m)}\]
with ${^{(n)}\! w} =w(p_1, \dots, p_{n-1}) \in AP_{n}$ and $w^{(m)}= \ w^{op}(q^{n+1}, \dots, q^{n+m-1}) \in AP_m$. Let 
\[ f_n: A \otimes k AP_{n+m} \otimes A \to A \otimes k AP_{n} \otimes A\]
 be defined by
\[f_n(1 \otimes w \otimes 1) =
\begin{cases}
1 \otimes 1 \otimes \hat f (w) & \mbox{if $n=0$},\\
 \sum_{\substack{\psi \in AP_{n} \\ L(\psi) \psi R(\psi) =   {^{(n)}\! w}   u }}  L(\psi) \otimes \psi \otimes R(\psi) \hat f(w^{(m)}) & \mbox{if $n>0$} .
\end{cases} \]

\begin{remark}\label{f par}
From Lemma \ref{particion} and Lemma \ref{dos} we can deduce that if $n$ is even then 
\[ f_n(1 \otimes w \otimes 1) = 1 \otimes {^{(n)}\! w}  \otimes u \  \hat f(w^{(m)}) \]
because $w(p_1, \cdots, p_n)=  {^{(n)}\! w} \   u \ b$ and hence 
$$\psi \in \Sub(w(p_1, \cdots, p_n))=\{ \psi_1, \psi_2=w(p_1, \dots, p_{n-1}) \}.$$
But $b$ is a non trivial path,  then $L(\psi) \psi R(\psi) =  {^{(n)}\! w} \   u$ implies that 
$\psi= \psi_2= {^{(n)}\! w}$.
\end{remark}

\begin{prop} Let $m>0$ and let $f \in \Hom_{A-A}( A\otimes k AP_{m} \otimes A,  A )$ be such that $\hat f \in \Ker F_{m+1}$.  Then $f= \mu f_0$ and $f_{n-1} d_{m+n} = d_n f_n$ for any $n \geq 1$.
\end{prop}

\begin{proof}
It is clear that $f= \mu f_0$ since for any $w \in AP_{m}$ we have that
\[ \mu f_0 (1 \otimes w \otimes 1) = \mu (1 \otimes 1 \otimes \hat f(w)) =\hat f(w) = f (1 \otimes w \otimes 1) . \]  Let $n \geq 1$ and let $w \in AP_{n+m}$. By Lemma \ref{nucleo} we have that $\hat f$ is a linear combination of basis elements in
$^-(0,0)_m^-$, $(1,0)_m^-$, $^-(0,1)_m $, $ (1,1)_m$ and $(id +(-1)^m \phi_{m})((1,0)_m^+$. The proof will be done in several steps considering $\hat f=f_{(\rho, \gamma)}$ with $(\rho, \gamma)$ belonging to each one of the previous sets.
\begin{itemize}
\item [(i)] Assume $(\rho, \gamma) \in {^-(0,0)_m^-}$. Using that $Q_1 \hat f(w^{(m)} )\subset I$ we have that
\[ f_n(1 \otimes w \otimes 1) =  \sum_{\substack{\psi \in AP_{n} \\  L(\psi) \psi R(\psi) =  {^{(n)}\! w}   u }}  L(\psi) \otimes \psi \otimes R(\psi) \hat f(w^{(m)}) =  L(\psi) \otimes \psi \otimes \hat f(w^{(m)}) \]
if there exists $\psi \in AP_{n}$ such that $L(\psi) \psi  =   {^{(n)}\! w}   u$ and zero otherwise.  In the first case
\begin{eqnarray*}
d_n f_n(1 \otimes w \otimes 1) & = &  \sum_{\substack{\phi \in AP_{n-1} \\  L(\phi) \phi R(\phi) = \psi} } L(\psi) L(\phi) \otimes \phi \otimes R(\phi)  \hat f(w^{(m)}) \\
 & = &   L(\psi) L(\phi) \otimes \phi\otimes \hat f(w^{(m)})
\end{eqnarray*}
for $\phi \in AP_{n-1}$ such that $L(\psi)  L(\phi) \phi = L(\psi)  \psi=  {^{(n)}\! w}   u$.
On the other hand, if $n+m$ is even,
\[f_{n-1} d_{n+m} (1 \otimes w \otimes 1) = f_{n-1} \left (\sum_{\substack{\psi'  \in \Sub(w)}} L(\psi' ) \otimes \psi'  \otimes R(\psi' ) \right ) =  L(\psi' ) f_{n-1} (1 \otimes \psi' \otimes 1) \]
with $\psi'$ such that $L(\psi') \psi' = w$ since $\hat f(\psi'^{(m)} )R(\psi') \subset \hat f(\psi'^{(m)} ) Q_1 \subset I$ for any $\psi'$ such that $\vert R(\psi') \vert >0$.   In case $n+m$ is odd we get the same final result.  Now  $Q_1 \hat f(\psi'^{(m)} )\subset I$ and $\psi'= {^{(n-1)}\!\psi'} u w^{(m)}$, then
$$L(\psi') f_{n-1} (1 \otimes \psi' \otimes 1) =   L(\psi') L(\phi') \otimes \phi' \otimes \hat f(w^{(m)})$$
if there exists $\phi' \in AP_{n-1} $ such that $L(\psi') L(\phi' ) \phi'  =  L(\psi') {^{(n-1)}\! \psi'}   u = {^{(n)}\! w}   u$ and zero otherwise.

The desired equality holds because: if $\psi$ and $\phi'$ do not exist, both terms vanish. 
If $\psi$ exists, then it is clear that $\phi'$ also exists, in fact $\phi' = \phi$ and  $L(\psi') L(\phi') \otimes \phi'  = L(\psi)  L(\phi) \otimes \phi$.  Finally assume that there is no $\psi \in AP_{n}$ such that $L(\psi) \psi  =   {^{(n)}\! w}   u$ and that $\phi'$ exists with $L(\psi') L(\phi' ) \in \P$.  If $n$ is even,  Lemma \ref{A}(i) applied on ${^{(n)}\! w}$  and $\phi'$ says that $t(p_1) \leq s(\phi')$ and hence $L(\psi') L(\phi' )=0$.  If $n$ is odd, Lemma \ref{A}(ii) applied on ${^{(n-1)}\! w } $ and $\phi'$ implies the existence of $\psi$, a contradiction.

\item [(ii)] Assume $(\rho, \gamma) \in (1,0)_m^-$.  Then $\rho= \alpha \hat \rho = \alpha \rho_1 \cdots \rho_s$, $\gamma= \alpha \hat  \gamma$ and $\alpha \rho_1 \in \R$.  Then
$f_n (1 \otimes w \otimes 1) = 0$ if $w^{(m) }\not = \rho$.  In this case $f_{n-1} d_{n+m} (1 \otimes w \otimes 1)$ also vanishes: the assertion is clear if $\rho$ does not divide $w$ and, if it does, $w=L(\rho) \rho R(\rho)$ with $\vert R(\rho)\vert >0$ and $\hat f(\rho) R(\rho)$ vanishes.
Assume now that $w^{(m)} = \rho$. This means that $w$ contains the relation $q^{n+1}=\alpha \rho_1$, by Lemma \ref{p=q} we have that $q^{n+1}=p_{n+1}$ and $q^{n}=p_n$,  and by Lemma \ref{cuadratica} we have that $\vert \Sub(w)\vert =2$. Then
\begin{eqnarray*}
f_{n-1} d_{n+m} (1 \otimes w \otimes 1) &=& L(\psi_1) f_{n-1} (1 \otimes \psi_1 \otimes 1) + (-1)^{n+m} f_{n-1} (1 \otimes \psi_2 \otimes 1)R(\psi_2) \\
&=& L(\psi_1) f_{n-1} (1 \otimes \psi_1 \otimes 1)
\end{eqnarray*}
since $\hat f(\psi_2^{(m)}) =0$, and $\psi_1= w^{op}(q^2, \cdots, q^{n+m-1})= {^{(n-1)}\! \psi_1}  \ v \ w^{(m)} =  {^{(n-1)}\! \psi_1}  \ v \ \rho$. 
If $n$ is odd, by Remark \ref{f par} we have that 
\[ f_{n-1} d_{n+m} (1 \otimes w \otimes 1) = L(\psi_1) \otimes {^{(n-1)}\! \psi_1}  \otimes  v \hat f(w^{(m)}) =  L(\psi_1) \otimes {^{(n-1)}\! \psi_1}  \otimes  v  \gamma , \]
and if $n$ is even
\[ f_{n-1} d_{n+m} (1 \otimes w \otimes 1) = \sum_{\substack{\phi \in AP_{n-1} \\  L(\phi) \phi R(\phi)= {^{(n-1)}\! \psi_1} v} }    L(\psi_1) L(\phi) \otimes \phi \otimes  R(\phi) \gamma.\]

On the other hand, $w=  {^{(n)}\! w} \ u \ w^{(m)} =  {^{(n)}\! w} \ u \ \rho$ and
\begin{eqnarray*}
  f_n (1 \otimes w \otimes 1) &=& \sum_{\substack{\psi \in AP_{n} \\  L(\psi) \psi R(\psi) =  {^{(n)}\! w} u  }}  L(\psi) \otimes \psi \otimes R(\psi) \hat f(w^{(m)}) \\
&=& \sum_{\substack{\psi \in AP_{n} \\   L(\psi) \psi R(\psi) =  {^{(n)}\! w}  u }}  L(\psi) \otimes \psi \otimes R(\psi) \ \alpha  \hat \gamma.
\end{eqnarray*}
By definition we have that $s(p_{n+1}) < t(p_n) < t(p_{n+1})$,  and $p_{n+1} = \alpha \rho_1$ so $t(p_n) = t(\alpha)$. Then ${^{(n)}\! w} \  u \ \alpha =  {^{(n+1)}\! w}$ and
\[ \{ \psi \in AP_{n}, \  L(\psi) \psi R(\psi) =  {^{(n)}\! w} u  \} =  \Sub(  {^{(n+1)}\! w} ) \setminus \{ \hat \psi \}  \]
where $\hat \psi = w^{op} (q^2,  \cdots, q^n)=  {^{(n-1)}\! \psi_1} \  v  \ \alpha$ and $ {^{(n+1)}\! w} =L(\psi_1) \hat \psi$. So
\[   f_n (1 \otimes w \otimes 1) =  d_{n+1} (1 \otimes  {^{(n+1)}\! w}  \otimes 1) \hat \gamma - (-1)^{n+1} L(\psi_1) \otimes \hat \psi \otimes \hat \gamma,\]
and hence
\[ d_n f_n  (1 \otimes w \otimes 1) = (-1)^{n} L(\psi_1) d_n (1 \otimes \hat \psi \otimes  1) \hat \gamma.\]
If $n$ is odd 
\[d_n (1 \otimes \hat \psi \otimes  1) = L(\hat \psi_1) \otimes \hat \psi_1 \otimes 1 - 1 \otimes \hat \psi_2 \otimes R(\hat \psi_2) \] 
with $\hat \psi_2 = {^{(n-1)}\! \psi_1} $, $R(\hat \psi_2) = v \alpha$ 
and $L(\psi_1)  L(\hat \psi_1) = 0$ because 
$$s(L(\psi_1) )= s(p_1) < t(p_1)  = t(q^1) \leq s(q^3) =s(\hat \psi_1) =  t(L(\hat \psi_1))$$
implies that $p_1$ divides $L(\psi_1)  L(\hat \psi_1) $. So 
\[ d_n f_n  (1 \otimes w \otimes 1) =  L(\psi_1) \otimes \hat \psi_2 \otimes R(\hat \psi_2)  \hat \gamma =  L(\psi_1) \otimes {^{(n-1)}\! \psi_1}  \otimes  v \gamma. \]
Finally, if $n$ is even 
\[ d_n f_n  (1 \otimes w \otimes 1) = \sum_{\substack{\phi \in AP_{n-1} \\ L(\phi) \phi R(\phi)= \hat \psi}} L(\psi_1) L(\phi) \otimes \phi \otimes  R(\phi) \hat \gamma\]
and the desired equality holds since 
\[ \{\phi \in AP_{n-1}:  L(\phi) \phi R(\phi)= {^{(n-1)}\! \psi_1} v \}  = \{\phi \in AP_{n-1}:   L(\phi) \phi R(\phi)= \hat \psi \} \setminus \{ \hat \psi_1 \}\]
and, as we have already seen, $L(\psi_1) L(\hat \psi_1)=0$.

\item [(iii)] If $(\rho, \gamma) \in (1,1)_m$ then $\rho=\alpha \hat \rho \beta = \alpha \rho_1 \dots \rho_s \beta$, $\gamma= \alpha \hat \gamma \beta$ and $\alpha \rho_1, \rho_s \beta \in \R$. Then
$f_n (1 \otimes w \otimes 1) = 0$ if $w^{(m) }\not = \rho$.  In this case $f_{n-1} d_{n+m} (1 \otimes w \otimes 1)$ also vanishes: the assertion is clear if $\rho$ does not divide $w$ and, if it does, Lemma \ref{cuadratica} says that $\vert \Sub (w) \vert = 2$ and then
\[ f_{n-1} d_{n+m} (1 \otimes w \otimes 1) = L(\psi_1) f_{n-1}(1 \otimes \psi_1 \otimes 1) + (-1)^{n+m} f_{n-1}(1 \otimes \psi_2 \otimes 1) R(\psi_2).\]
The first summand vanishes since $\psi_1^{(m)} = w^{(m)} \not = \rho$.  For the second one, observe that it vanishes if $\psi_2^{(m)} \not = \rho$. If $\psi_2^{(m)}= \rho$ then $\psi_2= w(p_1, \cdots, p_{n+m-2})$ with $p_{n+m-2} = \rho_s \beta$ and $p_{n+m-1}= \beta R(\psi_2)$.  Then $\hat f(\psi_2^{(m)}) R(\psi_2) = \alpha \hat \gamma \beta R(\psi_2)=0$.

Assume now that $w^{(m)} = \rho$. This means that $w$ contains the relation $q^{n+1}=\alpha \rho_1$ and the proof follows exactly as in (ii).

\item [(iv)] If $(\rho, \gamma) \in {^-(0,1)_m}$ then $\rho= \hat \rho \beta = \rho_1 \cdots \rho_s \beta$, $\gamma= \hat \gamma \beta$ and $\rho_s \beta \in \R$. Now $f_n (1 \otimes w \otimes 1) = 0$ if $w^{(m) }\not = \rho$.  In this case $f_{n-1} d_{n+m} (1 \otimes w \otimes 1)$ also vanishes:
the assertion is clear if $\rho$ does not divide $w$ and, if it does, let $\hat \psi \in \Sub(w)$ be such that $\hat \psi^{(m)} = \rho$. Then
\begin{eqnarray*}
f_{n-1} d_{n+m} (1 \otimes w \otimes 1) &=&  f_{n-1} (L(\hat \psi) \otimes \hat \psi \otimes R(\hat \psi) )\\
&=&  \sum_{\substack{\phi \in AP_{n-1} \\L(\phi) \phi R(\phi)=  {^{(n-1)}\! \hat \psi \hat u} }}
L(\hat \psi) L(\phi) \otimes \phi \otimes R(\phi) \gamma R(\hat \psi ).
\end{eqnarray*}
Now $Q_1 \gamma \subset I$, then 
\[ f_{n-1} d_{n+m} (1 \otimes w \otimes 1) = L(\hat \psi) L(\phi) \otimes \phi \otimes \gamma R(\hat \psi )\]
if there exists $\phi \in AP_{n-1}$ such that $ L(\phi) \phi =  {^{(n-1)}\! \hat \psi \hat u}$.  Since $w^{(m)} \not = \rho$, we have that $\vert R(\hat \psi) \vert >0$ and $t(\beta) = s(R(\hat \psi))$.  Now,  the relation $\rho_s \beta$ implies that $w$ contains a relation starting in $\beta$, and hence $\gamma R(\hat \psi) = \hat \gamma \beta R(\hat \psi) =0$.

Assume now that $w^{(m) }= \rho$.  Then
\[f_{n-1} d_{n+m} (1 \otimes w \otimes 1) =  L(\psi_1) f_{n-1} (1 \otimes \psi_1 \otimes 1) \]
since $\psi^{(m)} \not = \rho$ for any $\psi \in \Sub(w), \psi \not = \psi_1$.
Now 
\[  f_{n-1} (1 \otimes \psi_1 \otimes 1) = \sum_{\substack{\phi \in AP_{n-1} \\ L(\phi) \phi R(\phi) ={^{(n-1)}\!  \psi_1} v} }L(\phi) \otimes \phi \otimes R(\phi)  \gamma \]
but $Q_1 \gamma \subset I$ so
\[f_{n-1} d_{n+m} (1 \otimes w \otimes 1) = L(\psi_1) L(\phi) \otimes \phi \otimes \gamma \]
if there exists $\phi \in AP_{n-1}$ such that $L(\phi) \phi  ={^{(n-1)}\!  \psi_1} v$ and zero otherwise.

On the other hand, 
\begin{eqnarray*}
 d_n f_n (1 \otimes w \otimes 1) &=& d_n \left ( \sum_{\substack{\psi \in AP_n \\ L(\psi) \psi R(\psi) = {^{(n)}\!  w} u}}
L(\psi) \otimes \psi \otimes  R(\psi) \right ) \gamma \\
&=& L(\psi) d_n ( 1 \otimes \psi \otimes 1)   \gamma 
\end{eqnarray*}
if there exists $\psi \in AP_{n}$ such that $L(\psi) \psi =  {^{(n)}\!  w} u$ and zero otherwise. In the first case
\begin{eqnarray*}
 d_n f_n (1 \otimes w \otimes 1) &=&  L(\psi)  \sum_{\substack{\phi \in AP_{n-1} \\ L(\phi) \phi R(\phi) =\psi}} L(\phi) \otimes \phi \otimes R(\phi)  \gamma \\
& = & L(\psi)  L(\phi') \otimes \phi' \otimes \gamma
\end{eqnarray*}
with $L(\psi) \psi =  L(\psi)  L(\phi')  \phi' =      {^{(n)}\!  w} u$. 
If $\psi$ and $\phi$ do not exist, the desired equality is clear.  If $\psi$ exists, then it is clear that $\phi$ also exists, in fact $\phi = \phi'$ and $L(\psi_1) L(\phi) \otimes \phi  = L(\psi)  L(\phi') \otimes \phi' $.

Finally assume that there is no $\psi \in AP_{n}$ such that $L(\psi) \psi  =   {^{(n)}\! w}   u$ and that $\phi$ exists with $L(\psi_1) L(\phi) \in \P$.  If $n$ is even,  Lemma \ref{A}(i)  applied on ${^{(n)}\! w}$  and $\phi$ says that $t(p_1) \leq s(\phi)$ and hence $L(\psi_1) L(\phi )=0$.  If $n$ is odd, Lemma \ref{A}(ii) applied on ${^{(n-1)}\! w } $ and $\phi$ implies the existence of $\psi$, a contradiction. 

\item [(v)] Finally consider the case for the subset $(id +(-1)^{m} \phi_m)((1,0)_m^+$, that is,  $(\rho, \gamma) \in (1,1)$, $\rho = \alpha \hat \rho \beta, \gamma = \alpha \hat \gamma \beta$ and $f= f_{(\alpha \hat \rho, \alpha \hat \gamma)} + (-1)^m  f_{( \hat \rho \beta,  \hat \gamma \beta)}$. In this case $\alpha \rho_1, \rho_s \beta \in \R$.

Now $f_n (1 \otimes w \otimes 1) = 0$ if $w^{(m) }\not = \alpha \hat \rho, \hat \rho \beta$.  In this case $f_{n-1} d_{n+m} (1 \otimes w \otimes 1)$ also vanishes: the assertion is clear if neither $\alpha \hat \rho$ nor $\hat \rho \beta$ divide $w$.  Assume that one of them does, then $w$ contains the cuadratic relation $\alpha \rho_1$ or $\rho_s \beta$, and hence $\vert \Sub(w) \vert =2$.  So 
\begin{eqnarray*}
f_{n-1} d_{n+m} (1 \otimes w \otimes 1) &=& L(\psi_1) f_{n-1} (1 \otimes \psi_1 \otimes 1) + (-1)^{n+m} f_{n-1} (1 \otimes \psi_2 \otimes 1)R(\psi_2) \\
&=&  (-1)^{n+m} f_{n-1} (1 \otimes \psi_2 \otimes 1)R(\psi_2)
\end{eqnarray*}
since $\psi_1^{(m)} =   w^{(m) }\not = \alpha \hat \rho, \hat \rho \beta$. If $\psi_2^{(m)} \not = \alpha \hat \rho, \hat \rho \beta$, then $f_{n-1} (1 \otimes \psi_2 \otimes 1)=0$.
If $\psi_2^{(m)} = \alpha \hat \rho$ we have that $\beta$ does not divide $w$ since otherwise $w^{(m) } =  \hat \rho \beta$; in this case $\hat f(\psi_2^{(m)}) R(\psi_2)= \alpha   \hat \gamma R(\psi_2)=0$
because $\vert R(\psi_2) \vert >0$, $\hat \gamma \beta \not =0$ and $\beta$ is not the first arrow of $R(\psi_2)$.  If $\psi_2^{(m)} = \hat \rho \beta$ then $p_{n+m-2} = \rho_s \beta$ and $p_{n+m-1}= \beta R(\psi_2)$, so $\hat f(\psi_2^{(m)}) R(\psi_2)= \hat \gamma \beta R(\psi_2)=0$.

Assume now that $w^{(m)}= \alpha \hat \rho$.  Then $q^{n+1}= \alpha \rho_1 = p_{n+1}$ and
\begin{eqnarray*}
f_{n-1} d_{n+m} (1 \otimes w \otimes 1) &=& L(\psi_1) f_{n-1} (1 \otimes \psi_1 \otimes 1) + (-1)^{n+m} f_{n-1} (1 \otimes \psi_2 \otimes 1)R(\psi_2) \\
&=&  L(\psi_1) f_{n-1} (1 \otimes \psi_1 \otimes 1)
\end{eqnarray*}
since $\hat f(\psi_2^{(m)}) =0$.
Now the proof follows as in (ii).

If $w^{(m)}= \hat \rho \beta$ assume that $\alpha$ does not divide $w$, that is, $\psi_2^{(m)} \not = \alpha \hat \rho$. Then $q^{n+m-1}=  \rho_s \beta = p_{n+m -1}$ and
\begin{eqnarray*}
f_{n-1} d_{n+m} (1 \otimes w \otimes 1) &=& L(\psi_1) f_{n-1} (1 \otimes \psi_1 \otimes 1) + (-1)^{n+m} f_{n-1} (1 \otimes \psi_2 \otimes 1)R(\psi_2) \\
&=&  L(\psi_1) f_{n-1} (1 \otimes \psi_1 \otimes 1)
\end{eqnarray*}
since $\hat f(\psi_2^{(m)}) =0$.
Now the proof follows as in (iv).

Finally, assume that $w^{(m)}= \hat \rho \beta$ and that $\alpha$ divides $w$.  In this case $\psi_2^{(m)} = \alpha \hat \rho$, $p_n = \alpha \rho_1=q^n$, $p_{n+m-1}= \rho_r \beta = q^{n+m-1}$, $w= {^{(n)}\!  w} w^{(m)}$.  So
\[ d_n f_n (1 \otimes w \otimes 1) = d_n (1 \otimes  {^{(n)}\!  w} \otimes 1) \hat f(\hat \rho \beta) = (-1)^m  d_n (1 \otimes  {^{(n)}\!  w} \otimes 1) \hat \gamma \beta.\] 
On the other hand, $w$ contains a quadratic divisor so $\Sub(w) = \{ \psi_1, \psi_2\}$, and $R(\psi_2)=\beta$.  Then
\begin{eqnarray*}
f_{n-1} d_{n+m} (1 \otimes w \otimes 1) &=& L(\psi_1) f_{n-1} (1 \otimes \psi_1 \otimes 1) + (-1)^{n+m} f_{n-1} (1 \otimes \psi_2 \otimes 1)R(\psi_2) \\
&=&   L(\psi_1) \otimes {^{(n-1)}\!  \psi_1}  \otimes \hat f( \hat \rho \beta) \\
& + &  (-1)^{n+m}
\sum_{\substack{\phi \in AP_{n-1} \\L(\phi) \phi R(\phi)=  {^{(n-1)}\!  \psi_2  u} }}
 L(\phi) \otimes \phi \otimes R(\phi) \hat f(\alpha \hat \rho) \beta \\
&=&  (-1)^m L(\psi_1) \otimes {^{(n-1)}\!  \psi_1 } \otimes  \hat \gamma \beta \\
& + &  (-1)^{n+m}
\sum_{\substack{\phi \in AP_{n-1} \\L(\phi) \phi R(\phi)=  {^{(n-1)}\! \psi_2  u} }}
 L(\phi) \otimes \phi \otimes R(\phi)  \alpha \hat \gamma \beta .
\end{eqnarray*}
The equality follows since ${^{(n)}\!  w} = {^{(n-1)}\!  \psi_2  u} \alpha= L(\psi_1)  {^{(n-1)}\!  \psi_1} $ and
\[\Sub( {^{(n)}\!  w} ) = \{  {^{(n-1)}\!  \psi_1} \} \cup \{ \phi \in AP_{n-1}: L(\phi) \phi R(\phi)=  {^{(n-1)}\!  \psi_2  u} \}.\]
\end{itemize} \qed
\end{proof}

In order to describe the product $[g \cup f]$ we need to choose convenient representatives of the classes $[f]$ and $[g]$, see \cite{Bu}.
Given $f \in \Hom_{A-A} (A \otimes kAP_m \otimes A, A)$ we define $f^{\leq}$ and $f^{\geq}$ as follows: we start by considering basis elements $f_{(\rho, \gamma)}$
\[ f_{(\rho, \gamma)}^{\leq} = \begin{cases}
f_{(\rho, \gamma)} & \mbox{if $(\rho, \gamma) \in (0,0)_m$}, \\
(-1)^{m-1} f_{\phi_m (\rho, \gamma)} & \mbox{if $(\rho, \gamma) \in (1,0)_m^+$}, \\
(-1)^{m-1} f_{\psi_m (\rho, \gamma)} & \mbox{if $(\rho, \gamma) \in (1,0)_m^{-+}$}, \\
0 & \mbox{if $(\rho, \gamma) \in (1,0)_m^{--}$}, \\
f_{(\rho, \gamma)} & \mbox{if $(\rho, \gamma) \in (0,1)_m$}, \\
0 & \mbox{if $(\rho, \gamma) \in (1,1)_m$}
\end{cases} \]
and 
\[ f_{(\rho, \gamma)}^{\geq} = \begin{cases}
f_{(\rho, \gamma)} & \mbox{if $(\rho, \gamma) \in (0,0)_m$}, \\
(-1)^{m-1}f_{\phi^{-1}_m (\rho, \gamma)} & \mbox{if $(\rho, \gamma) \in {^+ \! (0,1)_m}$}, \\
(-1)^{m-1}f_{\psi^{-1}_m (\rho, \gamma)} & \mbox{if $(\rho, \gamma) \in {^{+-}\! (0, 1)_m}$}, \\
0 & \mbox{if $(\rho, \gamma) \in {^{--}\! (0,1)_m}$}, \\
f_{(\rho, \gamma)} & \mbox{if $(\rho, \gamma) \in (1,0)_m$}, \\
0 & \mbox{if $(\rho, \gamma) \in (1,1)_m$}
\end{cases} \]
and then we extend by linearity. 
By Lemma \ref{nucleo} we have that $f - f^{\leq}, f - f^{\geq} \in \Im F_{m}$ for any $m >0$, and hence $[f] = [f^{\leq}] = [f^{\geq}]$. Moreover, observe that $f^{\leq}$ is a linear combination of basis elements in $(0,0)_m \cup (0,1)_m$ and $f^{\geq}$ is a linear combination of basis elements in $(0,0)_m \cup (1,0)_m$.

\begin{teo}
If $A$ is a triangular string algebra and $n, m>0$  then $\HH^n(A) \cup \HH^m(A) =0$.
\end{teo}

\begin{proof}
Let $[f] \in \HH^m(A)$ and $[g] \in \HH^n(A)$.  We will show that $[ g^{\leq} \cup f^{\geq}] =0$. Let $w \in AP_{n+m}$, $w=  {^{(n)}\!  w}  \ u \ w^{(m)}$ then
\[g^{\leq} \cup f^{\geq} ( 1 \otimes w \otimes 1) = \sum_{\substack{ \psi \in AP_n \\ L(\psi) \psi R(\psi) =  {^{(n)}\!  w} \ u}} L(\psi) \hat g^\leq ( \psi ) R(\psi) \hat f^\geq (w^{(m)}).\]
Since $f \in \Ker F_m$ and $g \in \Ker F_n$, we know that $f$ and $g$ are linear combination of basis elements as described in Lemma \ref{nucleo}. Moreover, $f^\geq$ is a linear combination of basis elements in $^{-}\! (0,0)^- \cup (1,0)$ and $g^\leq$ is a linear combination of basis elements in $^{-}\! (0,0)^- \cup (0,1)$.

The vanishing of the previous computation is clear if $f^\geq$ or $g^\leq$ are basis elements associated to pairs in $^{-}\! (0,0)^-$ because for $n,m >0$ we have that $\vert   \hat g^\leq (\psi)  \vert >0$ and $ \vert  f^\geq (w^{(m)})  \vert >0$.  Finally, if $f^\geq$ is a basis element associated to a pair $(\alpha \rho, \alpha  \gamma) \in (1,0)$ and  $g^\leq$ is a basis element associated to a pair $( \rho'  \beta,  \gamma' \beta) \in (0,1)$ then we only have to consider the summand with $\psi =  \rho' \beta$ and $w^{(m)} = \alpha  \rho$. In this case $w$ verifies the following conditions: $p_{n+1} = q^{n+1} = \alpha \rho_1$, $ {^{(n)}\!  w} \ u$, and hence also $ {^{(n+1)}\!  w}$, contains the quadratic relation $\rho'_s \beta$, by Lemma \ref{cuadratica} the element ${^{(n+1)}\!  w}$ has exactly two divisors, one of them sharing the ending point with ${^{(n+1)}\!  w}$, so $\psi = {^{(n)}\!  w}$ and $p_n = \beta u \alpha$. Now the summand we are considering is 
\begin{eqnarray*}
L(\psi) \hat g^\leq ( \psi ) R(\psi) \hat f^\geq (w^{(m)}) & = & \hat g^\leq ( {^{(n)}\!  w} ) u \hat f^\geq (w^{(m)}) \\
& = &  \gamma' \beta u \alpha  \gamma \\
& = &  \gamma' p_n   \gamma = 0.
\end{eqnarray*}  \qed
\end{proof}

\end{document}